\documentclass[a4paper,11pt]{article}

\usepackage{amssymb,amsmath,amsthm,times,mathrsfs,fullpage}
\usepackage{tikz,wrapfig,multirow}
\usepackage{amsmath,amssymb,amsfonts,amsthm,
latexsym, amscd, amsfonts, epsfig, epic,mathrsfs,color,eucal,colordvi,eucal,float,xcolor}
\usepackage[pdftex]{hyperref}
\hypersetup{plainpages=True, pdfstartview=FitH, bookmarksopen=true,
colorlinks=true,linkcolor=blue,citecolor=red}

\pgfdeclarelayer{background}
\pgfdeclarelayer{foreground}
\pgfsetlayers{background,main,foreground}

\newcounter{mycount}
\theoremstyle{plain}
\newtheorem{theorem}[mycount]{Theorem}

\newtheorem{lemma}[mycount]{Lemma}
\newtheorem{proposition}[mycount]{Proposition}
\newtheorem{remark}{Remark}
\newtheorem{definition}{Definition}

\begin{document}
\title{Expected number of pattern and submap occurrences\\ in random planar maps}
\author{
    Guan-Ru Yu\\
    Faculty of Mathematics\\
    University of Vienna\\
    Vienna, 1090\\
    Austria}
\maketitle

\begin{abstract}
Drmota and Stufler proved recently that the expected number of pattern occurrences of a given map is asymptotically linear when the number of edges goes to infinity. In this paper we improve their result by means of a different method. Our method allows us to develop a systematic way for computing the explicit constant of the linear (main) term and shows that it is a positive rational number. Moreover, by extending our method, we also solve the corresponding problem of submap occurrences.
\end{abstract}

\section{Introduction and main results}

Within the last 20 years, there has been a recovering interest in planar maps which started mainly due to the work by Schaeffer \cite{GS1998} and culminated in the identification of the Brownian map as the scaling limit of large planar maps \cite{BJM2014,JFLG2013,GM2013}. 
Recently, there has also been a growing interest in local convergence (for example, in the case of quadrangulations  \cite{AS2003,MK2005}).
In particular, Stephenson \cite{RS2018} showed, with the help of the  Bouttier--Di Francesco--Guitter bijection \cite{BFG2004}, that random planar maps converge locally to a {\it uniform infinite planar map} (UIPM) $M_\infty$. 
This means that for every $\delta\ge 1$ and for every planar map $\widehat{m}$ of radius $\delta$ (measured from the root vertex) we have
\[
\lim_{n\to\infty} \mathbb{P}\{ U_\delta(M_n) = \widehat{m}\} =   \mathbb{P}\{ U_\delta(M_\infty) = \widehat{m}\},
\]
where $U_\delta(M_n)$ denotes the $\delta$-neighborhood of the root in
a planar map $M_n$ with $n$ edges.
The limiting object UIPM is only described with the help of a proper adaption of the Bouttier-Di Francesco-Guitter bijection to infinite trees (with a spine), see \cite{RS2018}.

Motivated by these results, we present some explicit results on the local convergence in Section~\ref{LPP}, which are based on a combinatorial approach. 
They do not give a full picture but complement the results from  \cite{RS2018} and constitute a direct and explicit analysis that can be also generalized to situations, where the approach from \cite{RS2018} cannot be applied, for example, to 2-connected maps.

On the other hand, instead of local problems,
Drmota and Stufler \cite{DS2020} considered planar maps adjusted with a (regular critical) Boltzmann distribution and showed that the expected number of  
the random number $t(\widehat{m},M_n)$ of occurrences of some planar map $\widehat{m}$
as a pattern in a random planar map (with $n$ edges) $M_n$ 
is asymptotically linear when $n$ goes to infinity, i.e.,
\[
\mathbb{E}[t(\widehat{m},M_n)]\sim c(\widehat{m})\cdot n \, ,
\]
for some constant $c(\widehat{m})>0$. Their proof is based on rerooting and is an extension of a formula by Liskovets \cite{VAL1999}.

The main purpose of this paper is to compute the constant $c(\widehat{m})$ of the result above and, moreover,
to show that it is a positive rational number. 
An exciting thing is that we can solve this 
with the help of our local convergence results which we mentioned before.
Moreover, another surprise is that we can use the same idea to solve
the corresponding problem of submap occurrences, too.

We now give a more detailed introduction and some definitions before we present our main results.
The study of planar maps goes back to Tutte \cite{WTT1963}.
A {\em planar map} is a connected planar graph (with loops and multiple edges allowed)
embedded in the plane. A map is {\em rooted} if one of its edges is distinguished and
directed. We call this oriented edge the {\em root edge}, the starting vertex of the
root edge the {\em root vertex} and the face to the left/right of the root edge is called the
{\em root face}/{\em second face}.
Throughout the paper, all the considered maps are rooted and planar, and their root faces are considered as outer faces. Moreover, we consider two maps isomorphic if they are isomorphic in the graph-theoretical sense after being embedded on the sphere.

When it comes to pattern occurrences in maps, the study can be classified into two categories.
One of them refers to local problems concerning pattern occurrences at the root,
such as determining the probability of a pattern to occur at the root of a random map. 
The other one is global which deals with patterns occurring in the whole map. Here the main goal is to find the distribution of the number of occurrences of a given pattern.

An {\em inner-face} of a map $M$ is a face in $M$ but not its root face. We denote by $F^*(M)$ the set of
all inner-faces of $M$.

\begin{definition}[pattern]\label{defpp}
A map $P$ occurs as a {\color{red} pattern at the root} of a map $M$ if
the vertex set $V(P)$ is a subset of the vertex set $V(M)$, 
the edge set $E(P)$ is a subset of the edge set $E(M)$,
the inner-face set $F^*(P)$ is a subset of the the inner-face set $F^*(M)$, 
and the root edge of $P$ is the root edge of $M$.
A map $P$ occurs as a {\color{red} pattern} in a map $M$ if
$V(P)\subseteq V(M)$, 
$E(P)\subseteq E(M)$,
and $F^*(P)\subseteq F^*(M)$.
\end{definition}
 
The conditions between local and global occurrence are almost the same,
the only difference is that the global one unties the location of the root edge,
see examples in Section~\ref{POPT}.

Our first main result is an improvement of the result by Drmota and Stufler from~\cite{DS2020}, 
where it was proved that that the mean number of pattern occurrences is asymptotically linear in the size of a map.

\begin{theorem}\label{thmpt}
Let $\widehat{m}$ be a given planar map and $n$ a positive integer. Then the expected number of occurrences of $\widehat{m}$ as a pattern in maps with $n$ edges is given by
\[
\mathbb{E}[t(\widehat{m},M_n)]= c_1\cdot n+c_2+O(1/n),
\]
where $c_1$ and $c_2$ are constants that depend on $\widehat{m}$ and, moreover, $c_1$ is a computable positive rational number.
\end{theorem}

This result is obtained from our explicit local results in Section~\ref{LPP} and
can also be extended to the corresponding problem of submap occurrences.

We say the face $f_1$ is {\em part of} the face $f_2$  
if $f_2$ is divided by some edges and vertices into several parts, and $f_1$ is chosen from one of them.

\begin{definition}[submap]
A map $S$ occurs as a {\color{red} submap at the root} of a map $M$ if
$V(S)\subseteq V(M)$, 
$E(S)\subseteq E(M)$,
and the root edge of $S$ is the root edge of $M$.
A map $S$ occurs as a {\color{red} submap} in a map $M$ if
$V(S)\subseteq V(M)$,
$E(S)\subseteq E(M)$,
and the root face of $M$ is not (or is not part of) the inner-face of $S$.
\end{definition}

Here we need to mention that Gao \cite{BGR1992, GW2003} uses the notion ``submap''  differently in his series of papers. More precisely, his definition of submaps is closer to our definition of patterns. 

Examples of submaps can be found in Section~\ref{SOPT}. The following is our second main result:

\begin{theorem}\label{thmsm}
Let $\widehat{m}$ be a given planar map and $n$ a positive integer. Then the expected number of occurrences of $\widehat{m}$ as a submap in maps with $n$ edges is given by
\[
\mathbb{E}[s(\widehat{m},M_n)]= c'_1\cdot n+c'_2+O(1/n),
\]
where $c'_1$ and $c'_2$ are constants that depend on $\widehat{m}$ and $c'_1$ is a computable positive rational number.
\end{theorem}

It is widely believed that the random variable that counts the number of occurrences of a pattern satisfies a central limit theorem although there are still few results so far.
The first central limit theorem concerning pattern occurrences was given in~\cite{DP2013}, where the authors studied the problem of enumerating faces of a given valency. 
Next, the enumeration of the double 3-gons \cite{DY2018} was provided, serving as the first solved problem that goes beyond a single-face pattern.
In this paper, we entirely solve the problem of finding the first moment of the distribution of the number of occurrences of a pattern/submap. Higher moments such as the variance and proving central limit theorems are still open problems. 


We conclude this section with a short plan of the paper. In Section~\ref{GFPM}, we review a classical result about counting maps via  building and studying functional equations with respect to adequate generating functions for maps that are constructed by removing root edges from bigger maps. This will serve as an essential tool in this paper.

In Section~\ref{PP} and~\ref{LPP}, we deal with the problems of occurrence of pure polygons and patterns at the root of maps, respectively. We show that counting maps with some fixed patterns at the root is actually equivalent to counting maps with root faces being pure polygons. This is the reason why we will solve the pure polygons problem before the local pattern occurrences problem.

In Section~\ref{POPT}, 
we show that the pattern occurrence counting problem in maps relies on the local pattern occurrences problem. 
More precisely, we will show a relation between them under the rerooting method. 
And in Section~\ref{SOPT}, 
we further extend the results from patterns to submaps by inserting maps into inner-faces of patterns with the rerooting method. 
We prove our main results in Section~\ref{PF} and give examples of computing the main term's constants of our two main results in Section~\ref{EX}.

\section{Generating function for planar maps}\label{GFPM}

In this section, we review some well known results from \cite{GJ1983} concerning enumeration of combinatorial objects via generating functions. Let $\mathscr{M}$ be the class of all rooted planar maps. We say that a map is {\em bridgeable} if its root edge is a \emph{bridge}, i.e., an edge whose deletion disconnects the map.
 By distinguishing the cases of the root edge in a map, we obtain the following equation:
\begin{align}\label{mbmn}
\mathscr{M}=\bullet+\mathscr{M}^{(b)}+\mathscr{M}^{(n)},
\end{align}
where $\mathscr{M}^{(b)}$ (resp. $\mathscr{M}^{(n)}$) represents the class of bridgeable (resp. non-bridgeable) maps and $\bullet$ corresponds to the case when the map has no edges. Now, let $M(z,u)$ be the bivariate generating function as follows:
\begin{align}\label{dmzu}
M(z,u):=\sum_{n,k \geq 0}{m_{n,k}\,z^nu^k},
\end{align}
where $m_{n,k}$ is the number of maps with $n$ edges and with root faces of valency $k$. By setting $u=1$, we obtain the following formula for $M(z,1)$:
\begin{align*}
M(z,1)=\sum_{n,k \geq 0}{m_{n,k}\,z^n1^k}=\sum_{n\geq 0}{\left(\sum_{k\geq 0}{m_{n,k}}\right)z^n}.
\end{align*}
From this it is clear that
\begin{align}\label{m0-gen2}
m_n:=\sum_{k\geq 0}{m_{n,k}} 
\end{align}
is the number of planar rooted maps with $n$ edges. 

By deleting root edges and studying the so constructed maps, one obtains the following functional equation involving the generating function from (\ref{dmzu}):
\begin{align}\label{m1}
M(z,u)=1+zu^2M(z,u)^2+zu\frac{M(z,1)-uM(z,u)}{1-u}.
\end{align}
This functional equation relates two unknown functions $M(z,u)$ and $M(z,1)$.
The so-called \emph{quadratic method} is a standard procedure to solve such equations.
We can rewrite (\ref{m1}) as
\begin{align*}
\left(M(z,u)-\frac{u^2z-u+1}{2u^2(1-u)z}\right)^2=
\frac{u^4z^2-2u^2(u-1)(2u-1)z+1-u^2}{4u^4(1-u)^2z^2}+\frac{M(z,1)}{u(1-u)}.
\end{align*}
Hence, if we bind $u$ and $z$ (by setting $u = u(z)$) in such a way that the left-hand side of the above equation vanishes, then the right-hand side also vanishes, and so does the partial derivative with respect to $u$ of the right-hand side. Thus, these two relations can be used to obtain
\begin{align}\label{m54}
u(z) = \frac{5-\sqrt{1-12z}}{2(z+2)}
\qquad\text{and}\qquad
M(z,1)=\frac{18z-1+(1-12z)^{3/2}}{54z^2}.
\end{align}
We now have the singular expansion of $M(z,1)$ at its dominant singularity $z=1/12$ as follows:
\begin{align}\label{m1i2}
M(z,1)=\sum_{i=0}^{\infty}a_i(1-12z)^{i/2}
\end{align}
and obtain that $a_1=0$ and
\begin{align}\label{exm0}
M(z,1)=\,\frac{4}{3}-\frac{4}{3}\,(1-12z)+\frac{8}{3}\,(1-12z)^{3/2}
-4\,(1-12z)^2+\frac{16}{3}\,(1-12z)^{5/2}+\cdots.
\end{align}

Next, let us recall the {\it Transfer Theorem} from \cite{FO1990}\cite[Theorem VI.1]{FS2009}:
Let $\alpha$ be an arbitrary complex number in $\mathbb{C}\setminus \mathbb{Z}^+$.
The coefficient of $z^n$ in $f(z)=(1-z)^{-\alpha}$
admits for large $n$ a complete asymptotic expansion in descending powers of $n$, e.g.,
\begin{align*}
[z^n]f(z)=\frac{n^{\alpha -1}}{\Gamma (\alpha)}\bigg(1+\frac{\alpha(\alpha-1)}{2n}
+\frac{\alpha(\alpha-1)(\alpha-2)(3\alpha-1)}{24n^2}
+O(n^{-3})\bigg).
\end{align*}

Thus, we know that the $3/2$-term is the main term of Equation (\ref{exm0}). This leads to
\begin{align}\label{ma3b}
m_n=[z^n]M(z,1)\sim \frac{a_3}{\Gamma(-3/2)} n^{-5/2}12^n=\frac{2}{\sqrt{\pi}}n^{-5/2}12^n.
\end{align}
Note that one can also obtain an explicit number by Lagrange inversion theorem as follows:
\[
m_n=\frac{2(2n)!}{(n+2)!n!}3^n.
\]
Moreover, we are also interested in the asymptotic behavior of $M(z,u)$ which allows us to study the limiting distribution of the valency of the root face. We first show that   
\begin{align}\label{mu3b}
[z^n]M(z,u)\sim \frac{a_3(u)}{\Gamma(-3/2)} n^{-5/2}12^n.
\end{align}
where
\begin{align}\label{a3ueue}
a_3(u)=\frac{8u}{\sqrt{3(u+2)(-5u+6)^3}}.
\end{align}

By plugging the result from Equation (\ref{m54}) into Equation (\ref{m1}),
we have, for $u$ close to $1$, that $M(z,u)$ has the same singularity as $M(z,1)$ and its singular expansion is as follows:
\begin{align}\label{mui2}
M(z,u)=\sum_{i=0}^{\infty}a_i(u)(1-12z)^{i/2},
\end{align}
where each $a_i(u)$ is an analytic function and satisfies
\begin{align*}
a_i=\lim_{u\rightarrow 1}a_i(u),
\end{align*}
which can be checked by comparing (\ref{m1i2}) and (\ref{mui2}).

We now claim that $a_1(u)=0$ and $a_3(u)\ne 0$ for $u$ close to $1$,
which implies that $a_3(u)(1-12z)^{3/2}$ is the main term of $M(z,u)$.
To prove the claim, we first define $Z:=(1-12z)^{1/2}$ and then rewrite $z$, $M(z,1)$, and $M(z,u)$ by
\[
z=\frac{1}{12}-\frac{1}{12}Z^2,\quad\quad M(z,1)=\frac{4}{3}-\frac{4}{3}Z^2+\frac{8}{3}Z^3+\cdots,
\]
and
\[
M(z,u)=a_0(u)+a_1(u)Z+a_2(u)Z^2+a_3(u)Z^3+\cdots.
\]
We put everything above back into Equation (\ref{m1}) and compare the coefficients of $Z$ on both sides. 
First, we compute the constant term coefficients on both sides and obtain the following equation:
\[
a_0(u)=1+\frac{u^2}{12}a_0(u)^2+\frac{u\left(a_0-u a_0(u)\right)}{12(1-u)},
\]
where $a_0(u)$ should have two solutions. We choose the one that satisfies $\lim_{u\to 1}a_0(u)=a_0=4/3$ and get
\begin{align}\label{a0ueue}
a_0(u)=\frac{-3u^2+36u-36+\sqrt{3(u+2)(-5u+6)^3}}{6u^2(u-1)}.
\end{align}
Next, we check the coefficients of $Z^1$ on both sides. We have
\[
a_1(u)=\frac{1}{6}u^2a_0(u)a_1(u)-\frac{u^2a_1(u)}{12(1-u)},
\]
which leads to $a_1(u)=0$.

The only thing left is checking that $a_3(u)\ne 0$. We now compare the coefficients of $Z^3$ on both sides of Equation (\ref{m1}). 
From $a_1(u)=0$, we obtain that 
\[
a_3(u)=\frac{u^2}{12}2a_0(u)a_3(u)+\frac{u\left(a_3-u a_3(u)\right)}{12(1-u)},
\]
and after inserting the result of $a_0(u)$ from (\ref{a0ueue}), we get that $a_3(u)$ has the form from (\ref{a3ueue}) and is equal to $8/3$ for $u$ close to $1$. Now, since
$a_3(u)$ is the coefficient of the $3/2$-term which is the main term of $M(z,u)$, we obtain the asymptotic behavior as claimed in (\ref{mu3b}). 

Now, we are able to obtain the limiting probability of the valency of the root face.
First, we know that
for a random map with $n$ edges, the probability that the root face has valency $k$  is
\begin{align*}
p_{n,k}:=\frac{m_{n,k}}{m_n}.
\end{align*}
Hence the probability generating function (PGF) of the root face valency is
\begin{align*}
p_n(u):=\sum_{k\ge 0}{p_{n,k}\,u^k}=\frac{1}{m_n}\sum_{k\ge 0}{m_{n,k}\,u^k}=\frac{[z^n]M(z,u)}{[z^n]M(z,1)},
\end{align*}
and we denote by $p^*_k$ the limit of $(p_{n,k})_{n\ge 0}$ as follows:
\begin{align*}\label{ppkk}
p^*_k=\lim_{n\to \infty}p_{n,k}.
\end{align*}
Then we have that 
the probability generating function of the limiting distribution of the root face valency is 
\[
p(u):=\sum_{k\ge 0}{p^*_k\,u^k}=\lim_{n\to \infty}\sum_{k\ge 0}{p_{n,k}\,u^k}
=\lim_{n\to \infty}\frac{[z^n]M(z,u)}{[z^n]M(z,1)},
\]
and by using the results from (\ref{ma3b}) and (\ref{mu3b}), we have
\[
p(u)=\frac{a_3(u)}{a_3}=\frac{\sqrt{3}u}{\sqrt{(u+2)(-5u+6)^3}}
\]
which has a dominant singularity at $u=6/5$. 
\begin{remark}
As $k\to \infty$, 
by the Transfer Theorem, we have
\begin{align*}
p^*_k=[u^k]p(u)\sim \frac{k^{1/2}}{2\sqrt{10\pi}} \left(\frac{5}{6}\right)^k.
\end{align*}
\end{remark}
Moreover, we can also show that $p^*_k$ is a positive rational number for each positive integer $k$.
From
\begin{align}\label{pue}
p(u)=&\frac{\sqrt{3}u}{\sqrt{(u+2)(-5u+6)^3}}\nonumber\\ 
=&\sqrt{3}u\left(\sum_{i\ge 0}\binom{-1/2}{i}\frac{1}{\sqrt{2}}\left(\frac{1}{2}u\right)^i\right)
\cdot\left(\sum_{j\ge 0}\binom{-3/2}{j}\frac{1}{6\sqrt{6}}\left(\frac{-5}{6}u\right)^j\right)\nonumber \\
=&\frac{u}{12}\left(\sum_{i\ge 0}\binom{2i}{i}\left(\frac{-1}{8}u\right)^i\right)
\cdot\left(\sum_{j\ge 0}\binom{2j}{j}(2j+1)\left(\frac{5}{24}u\right)^j\right),
\end{align}
we see that $p^*_k=[u^k]p(u)$ is rational. 
And as for the positivity of $p_k^*$, 
when $n\ge k$,
we can easily construct a subclass of the class of maps with $n$ edges and its root face of valency $k$
by putting a $k$-cycle into the root face of each map with $n-k$ edges 
such that the $k$-cycle is the root face of the resulting map.
Therefore, we have, for $n\ge k$,
\begin{align}\label{argu1}
m_{n,k}\ge m_{n-k}, \quad\text{which leads to}\quad
p_{n,k}=\frac{m_{n,k}}{m_n}\ge \frac{m_{n-k}}{m_n}.
\end{align}
Hence, when $n$ tends to infinity,
\[
p_k^*=\lim_{n\to \infty}p_{n,k}\ge \lim_{n\to \infty}\frac{m_{n-k}}{m_n}
=12^{-k}>0.
\]

\section{Pure polygons}\label{PP}
We say that a face is a {\em pure $\ell$-gon} ($\ell\ge 2$) if it is incident to exactly $\ell$ different edges and $\ell$ different vertices. Let $f_{\ell,n}$ be the number of maps with $n$ edges and whose root faces are pure $\ell$-gons. We denote by $\xi_{\ell,n}$ the probability that the root face in a map of size $n$
is a pure $\ell$-gon, i.e.,
\begin{align*}
\xi_{\ell,n}:=\frac{f_{\ell,n}}{m_n}.
\end{align*}

\begin{proposition}\label{tgn}
The limiting distribution of the probability that the root face in a random map
is a pure $\ell$-gon is
\begin{align}\label{zeqq}
\xi_{\ell}:=\lim_{n\to \infty}\xi_{\ell,n}=\frac{1}{12^{\ell}(\ell-1)!} \left(\frac{\partial^{\ell-1}}{\partial u^{\ell-1}}\frac{\sqrt{3}u^{\ell}}{\sqrt{(u+2)(-5u+6)^3}}\right)\Bigg|_{u=1},
\end{align}
and $\xi_{\ell}$ is a positive rational number.
\end{proposition}

\begin{remark}
As $\ell\to \infty$, we have 
\[
\xi_{\ell}\sim c\cdot {\ell}^{1/2}\left(\frac{5}{6}\right)^{\ell},
\] 
where $c$ is a suitable constant.
\end{remark}

Before we prove Proposition~\ref{tgn}, we present some lemmas.

\begin{lemma}\label{lemfz}
Let $F_{\ell}(z)$ be the ordinary generating function corresponding to $\big( f_{\ell,n} \big)_{n \geq 0}$. We have that,
for every $\ell \geq 2$,
\begin{align}\label{fmz}
F_{\ell}(z):=\sum_{n\geq 0}f_{\ell,n}\,z^n=\frac{z^\ell}{(\ell-1)!}\frac{\partial^{\ell-1}}{\partial u^{\ell-1}}u^{\ell-1}M(z,u)\bigg|_{u=1},
\end{align}
where $M(z,u)$ is defined by \eqref{dmzu}.
\end{lemma}

It is worth mentioning that three different proofs of this theorem are provided in~\cite{GRY2019}. Here we only present one of them, which is done by means of the \emph{star-inserting method}.
However, before proving Lemma~\ref{lemfz}, we recall some notions from graph theory and describe the construction used in the proof.

We say that a vertex $v$ is an {\em $\ell$-pie} ($\ell\ge 2$) if $v$ is incident to exactly $\ell$ different edges and $\ell$ different faces. By graph duality, a pure $\ell$-gon (a face) in a map corresponds to an $\ell$-pie (a vertex) in the dual map and vice versa.
Thus, the problem of counting maps with root faces being pure polygons is equivalent to the problem of counting maps with root vertices being pies.

A {\em star} is a tree with at least two edges in which there is only one internal node (and so the other vertices are leaves). Furthermore, an {\em $\ell$-star} is a tree with $\ell$ leaves. In fact, if we have a map whose root vertex is an $\ell$-pie, then we can decompose this map by deleting its root vertex $v$ and the $\ell$ edges that are incident to $v$. 
Consequently, the map is divided into two parts: an $\ell$-star and a remaining map.

The idea of the star-inserting method consists in inserting an $\ell$-star into the root face of a map to create an $\ell$-pie. Hence, we aim at showing that this decomposition is reversible. 
This will give us a relation between maps with roots being $\ell$-pies and maps of size smaller than $\ell$. This relation will turn out to be crucial in the proof of Lemma~\ref{lemfz}. The star-insertion consists of the following two steps:
First, we take a randomly chosen map $M$ and attach an $\ell$-star (making its internal node the root vertex) in such a way that the end of the root edge $e$ of the $\ell$-star has its end at the root vertex of $M$ and we let $e$ be the new root edge in the so constructed map.
Next, we connect the remaining $\ell-1$ edges with the vertices lying on the root face of $M$ in a random, yet crossing-free, way.
In the end, we obtain a new map whose root vertex (the internal node of the $\ell$-star) is an $\ell$-pie, as it is incident to $\ell$ different faces and $\ell$ different edges. See Figure~\ref{fig:star} for an example.

\begin{center}
\begin{figure}[H]
\begin{minipage}{1\textwidth}
    \begin{center}
        {\includegraphics[width=1\textwidth]{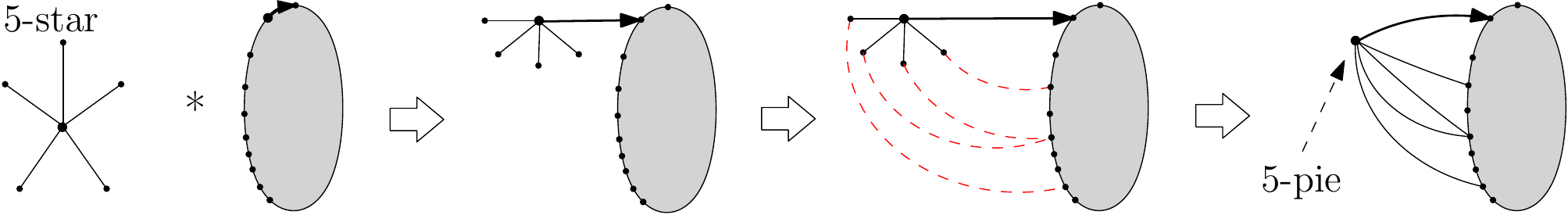}}
        \end{center}
\end{minipage}
\caption{Inserting a $5$-star in order to create a $5$-pie.}
\label{fig:star}
\end{figure}
\end{center}

Now, we prove Lemma~\ref{lemfz} by applying the above construction.

\begin{proof}[Proof of Lemma~\ref{lemfz}]
If a map $M$ has its root face of valency $k$, then there are $\frac{(\ell-1+k)!}{(\ell-1)!k!}$ different ways to connect $\ell-1$ non-root edges of the $\ell$-star to the vertices that belong to the root face of $M$ without crossings. 

The effect of adding an $\ell$-star into a map with $n$ edges and the root face of valency $k$ has the corresponding transformation:
\begin{align}\label{eqc1}
z^nu^k\mapsto z^{n+\ell}\,\frac{(\ell-1+k)!}{(\ell-1)!k!},
\end{align}
where the domain depends on $u$, but the image does not depend on it. However, the image depends on the power of $u$ in the domain.

By the following equation:
\begin{align*}
\frac{(\ell-1+k)!}{k!}=\frac{\partial^{\ell-1}}{\partial u^{\ell-1}}u^{\ell-1+k}\bigg|_{u=1},
\end{align*}
along with (\ref{eqc1}),
we have that each monomial $h(z,u)=z^nu^k$ has the following transformation:
\begin{align*}
h(z,u)\mapsto \frac{z^{\ell}}{(\ell-1)!}\frac{\partial^{\ell-1}}{\partial u^{\ell-1}}u^{\ell-1}h(z,u)\bigg|_{u=1},
\end{align*}
which means that if we insert an $\ell$-star into the root face of some map, we have the following transformation:
\[
M(z,u)\mapsto \frac{z^{\ell}}{(\ell-1)!}\frac{\partial^{\ell-1}}{\partial u^{\ell-1}}u^{\ell-1}M(z,u)\bigg|_{u=1}.
\]
Thus, the ordinary generating function $F_{\ell}(z)$ involves a differential operator as follows:
\begin{align*}
F_{\ell}(z)=\sum_{n\geq 0}f_{\ell,n}\,z^n=\frac{z^{\ell}}{(\ell-1)!}\frac{\partial^{\ell-1}}{\partial u^{\ell-1}}u^{\ell-1}M(z,u)\bigg|_{u=1},
\end{align*}
and the proof is finished.
\end{proof}

\begin{lemma}\label{posi1}
For every $\ell\ge 1$, we have that $\xi_{\ell}$ is positive.
\end{lemma}
\begin{proof}
In order to prove the positivity of $\xi_{\ell}$, we simply repeat the construction of showing the positivity of $p^*_{\ell}$. Actually, the construction gives us a sub-class of the class of maps whose root faces are not only faces of valency $\ell$ but also pure $\ell$-gons.
Thus, by the same argument as (\ref{argu1}), we have, for $n\ge \ell$,
\begin{align*}
f_{\ell,n}\ge m_{n-\ell}, \quad \text{which leads to } \quad
\xi_{\ell,n}=\frac{f_{\ell,n}}{m_n}\ge \frac{m_{n-\ell}}{m_n}.
\end{align*}
Hence, when $n$ tends to infinity, we have
\[
\xi_{\ell}=\lim_{n\to \infty}\xi_{\ell,n}\ge \lim_{n\to \infty}\frac{m_{n-\ell}}{m_n}=12^{-\ell}>0.\qedhere
\]
\end{proof}

Now, we are ready to prove Proposition~\ref{tgn}.

\begin{proof}[Proof of Proposition~\ref{tgn}]
By (\ref{mui2}), we have
\begin{align*}
\frac{\partial^{\ell-1}}{\partial u^{\ell-1}}\left(u^{\ell-1}M(z,u)\right)=
\sum_{i\ge 0}\frac{\partial^{\ell-1}}{\partial u^{\ell-1}}\left(u^{\ell-1} a_i(u)\right)(1-12z)^{i/2},
\end{align*}
where $a_1(u)=0$. Moreover, we have the following expansion of $z^{\ell}$:
\begin{align*}
z^{\ell}=\sum_{i=0}^{\ell}\binom{\ell}{i} \frac{(-1)^i}{12^{\ell}}(1-12z)^i. 
\end{align*}

Thus, we have for the expansion of $F_{\ell}(z)$ at its singularity,
\begin{align}\label{hm1}
F_{\ell}(z)=\sum_{i\ge 0}\kappa_{\ell,i}\,(1-12z)^{i/2},
\end{align}
with $\kappa_{\ell,1}=0$ and
\begin{align}\label{hm2}
\kappa_{\ell,3}=\frac{1}{12^{\ell}(\ell-1)!}\left(\frac{\partial^{\ell-1}}{\partial u^{\ell-1}}u^{\ell-1}a_3(u)\right)\Bigg|_{u=1}.
\end{align}
By the Transfer Theorem, we therefore have that
\begin{align*}
[z^n]F_{\ell}(z)\sim  \frac{\kappa_{\ell,3}}{\Gamma(-3/2)}\,n^{-5/2}12^n,
\end{align*}
and by comparing this with (\ref{ma3b}),
we obtain that
\begin{align}\label{topo}
\xi_{\ell}=&\lim_{n\to \infty}\frac{[z^n]F_{\ell}(z)}{[z^n]M(z,1)}=\frac{\kappa_{\ell,3}}{a_3}
=\frac{1}{12^{\ell}(\ell-1)!} \left(\frac{\partial^{\ell-1}}{\partial u^{\ell-1}}u^{\ell-1} \frac{a_3(u)}{a_3}\right)\bigg|_{u=1}
\end{align}
and by (\ref{a3ueue}), we obtain (\ref{zeqq}). Next, we show that $\xi_{\ell}\in \mathbb{Q}$. By 
replacing $u$ by $v+1$, we have

\begin{align*}
\xi_{\ell}&=\frac{1}{12^{\ell}(\ell-1)!} \left(\frac{\partial^{\ell-1}}{\partial v^{\ell-1}}\frac{\sqrt{3}(1+v)^{\ell}}{\sqrt{(v+3)(-5v+1)^3}}\right)\Bigg|_{v=0}
=\frac{1}{12^{\ell}}[v^{\ell-1}]\left(\frac{\sqrt{3}(1+v)^{\ell}}{\sqrt{(v+3)(-5v+1)^3}}\right),
\end{align*}
which can be rewritten to
\begin{align*}
\xi_{\ell}&=12^{-\ell}[v^{\ell-1}]\left(\sqrt{3}(1+v)^{\ell}(v+3)^{-1/2}(-5v+1)^{-3/2}\right)\\
&=12^{-\ell}[v^{\ell-1}]\left((1+v)^{\ell}(1+v/3)^{-1/2}(1-5v)^{-3/2}\right)\\
&=12^{-\ell}\sum_{j=0}^{\ell-1}\sum_{i=0}^{\ell-1-j}\binom{\ell}{\ell-1-i-j}\binom{-1/2}{i}\binom{-3/2}{j}
3^{-i}(-5)^{j}.
\end{align*}
This shows that $\xi_{\ell}$ is rational. Finally, by combining with Lemma~\ref{posi1}, the proof is finished. 
\end{proof}

\section{Local pattern probability}\label{LPP}
In this section, we provide a combinatorial approach to determine limiting probabilities of rooted pattern occurrences in random planar maps.
In particular, as the main goal in this section we will solve the following problem. Fix a local structure at the root and show that the limiting probability that this structure occurs exists and compute it.

An edge $e$ is called an {\em inner-edge} if $e$ is not incident to the root face, while $e$ is called an {\em outer-edge} if $e$ is a bridge and incident to the root face
(see examples in Figure~\ref{p_i}).
The following proposition is the main result of this section which gives explicit results on the local convergence:

\begin{proposition}\label{patthm}
The limiting probability $P_{\widehat{m}}$ that a pattern $\widehat{m}$ with $k$ inner-edges, $s$ outer-edges, and the root face of valency $\ell$ occurs at the root in a random map is given by
\begin{align*}
P_{\widehat{m}}=\xi_{\ell}\left(\frac{1}{12}\right)^{k-s},
\end{align*}
where $\xi_{\ell}$ is the limiting distribution of the probability that the root face in a random map is a pure $\ell$-gon given in (\ref{zeqq}).
\end{proposition}

As a matter of fact, Proposition~\ref{patthm} concerning the limiting distribution comes from Lemma~\ref{bij2} presented below, which constitutes an extension of Lemma~\ref{lemfz}.

Let $\tilde{f}_{\widehat{m},n}$ be the number of maps with $n$ edges and pattern $\widehat{m}$ occurring at the root and denote the corresponding generation function of $\big( \tilde{f}_{\widehat{m},n} \big)_{n \geq 0}$ by $\mathbb{F}_{\widehat{m}}(z)$, i.e.,
\begin{align*}
\mathbb{F}_{\widehat{m}}(z):=\sum_{n\ge 0}\tilde{f}_{\widehat{m},n}\,z^n.
\end{align*}

\begin{lemma}\label{bij2}
Let $\widehat{m}$ be a pattern with $k$ inner-edges, $s$ outer-edges, and the outer face of valency $\ell$. Then, we have
\begin{align*}
\mathbb{F}_{\widehat{m}}(z)=z^{k-s}\,F_{\ell}(z),
\end{align*}
where $F_{\ell}(z)$ is the ordinary generation function of the number of maps with the root faces being pure $\ell$-gons given in (\ref{fmz}).
\end{lemma}

We first present the idea of the construction for proving Lemma~\ref{bij2}.
First align the root edges of these two maps and then merge $\ell$ edges from the outer face of the pattern and $\ell$ edges of the second face of the map. 
Then, these two maps will be merged into a new map with the fixed pattern occurring at its root.

Note that it seems that we just merge those $\ell$ edges by keeping the relative position and put those inner-edges into the suitable location. However, we need to be careful of cases where the patterns have some outer-edges.
In fact, a fixed pattern with root face valency $\ell$ and $s$ outer-edges only has $\ell-s$ different edges on its boundary. So when we embed a fixed pattern with boundary length $\ell$, $k$ inner-edges and $s$ outer-edges into the pure $\ell$-gon, we have to take out $s$ edges (merge $s$ pairs of edges) and put $k$ edges (corresponding inner-edges) into this face.

We now give an example (see Figure \ref{p_i}), the pattern $P$ has boundary length $8$ and an inner-edge and an outer-edge.
We can put one edge into the map $M'$ whose second face is a pure $8$-gon and merge a pair of edges (because $e_2$ is an outer-edges) such that the pattern $P$ occurs at the root of the resulting map $M$.

\begin{center}
\begin{figure}[H]
\begin{minipage}{1\textwidth}
    \begin{center}
        {\includegraphics[width=0.95\textwidth]{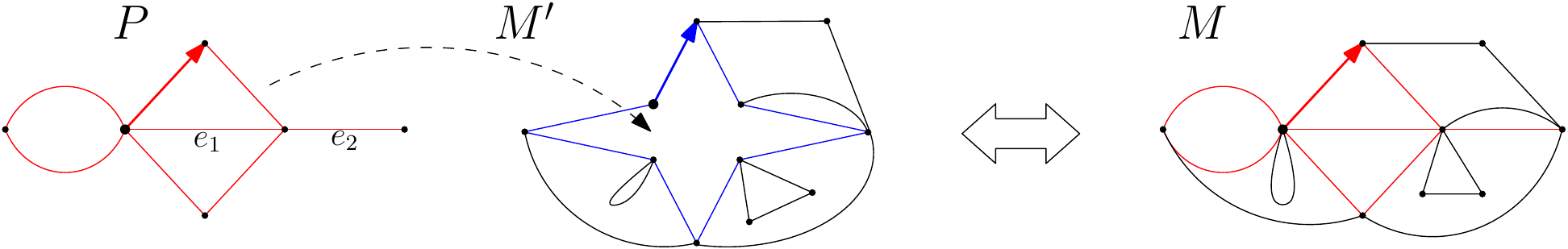}}
        \end{center}
\end{minipage}
\caption{The map $P$ has its root face of valency $8$. The second face of the map $M'$ is a pure $8$-gon.
The map $P$ is a pattern occurring at the root of the map $M$.
The edge $e_1$ is an inner-edge of the map $P$ and the edge $e_2$ is its outer-edge.}
\label{p_i}
\end{figure}
\end{center}

\begin{proof}[Proof of Lemma~\ref{bij2}] 
Suppose $\widehat{m}$ is a pattern with outer face valency $\ell$, $k$ inner-edges, and $s$ outer-edges. By following the construction above, 
we know that a map with $\widehat{m}$ at the root can be constructed by the map with the second face a pure $\ell$-gon. 
And actually, each map with $\widehat{m}$ at the root has a unique way to be constructed by its corresponding map whose second face is a pure $\ell$-gon.
This means that there is a bijection between them (for example $M$ and $M'$ in Figure \ref{p_i}). 
We therefore have that the number of maps with $n$ edges and whose second faces are pure $\ell$-gons is the same as the number of maps with $n+k-s$ edges 
and with $\widehat{m}$ occuring at the root. 
We can easily reverse roles of the root face and the second face in one map by changing the direction of its root edge, this means the number of maps with $n$ edges and whose second faces are pure $\ell$-gons is also $f_{\ell,n}$. Summarizing everything above, we obtain that
\[
\tilde{f}_{\widehat{m},n+k-s}=f_{\ell,n},
\]
and for the ordinary generating function, we have
\[
\mathbb{F}_{\widehat{m}}(z)=\sum_{n\ge 0}\tilde{f}_{\widehat{m},n+k-s}\,z^{n+k-s}=\sum_{n\ge 0}f_{\ell,n}\,z^{n+k-s}=z^{k-s}\, F_{\ell}(z).
\]
This concludes the proof.
\end{proof}

We again take the map $P$ in Figure~\ref{p_i} as an example. When considering the pattern $P$, we easily check that each map $M$ corresponds to a unique map $M'$ and vice versa, and owing to $k=s=1$, we have  $\tilde{f}_{P,n+1-1}=f_{8,n}$ which leads to $\mathbb{F}_P(z)=F_{8}(z)$.

\begin{proof}[Proof of Proposition~\ref{patthm}] 

For a given pattern $\widehat{m}$ with outer face valency $\ell$, $k$ inner-edges and $s$ outer-edges, the probability that this pattern occurs at the root of a random map with $n$ edges is equal to 
\begin{align*}
P_{\widehat{m},n}=\frac{\tilde{f}_{\widehat{m},n}}{m_n}=\frac{[z^{n}]\mathbb{F}_{\widehat{m}}(z)}{[z^n]M(z,1)}=\frac{[z^n]z^{k-s}F_{\ell}(z)}{[z^n]M(z,1)}.
\end{align*}
By Lemma~\ref{bij2} and as $n\to\infty$, 
we have the following limiting probability that
\begin{align*}
P_{\widehat{m}}:=&\lim_{n\to\infty}P_{\widehat{m},n}
=\lim_{n\to\infty}\frac{[z^{n}]\mathbb{F}_{\widehat{m}}(z)}{[z^n]M(z,1)}
=\lim_{n\to\infty}\frac{[z^n]z^{k-s}F_{\ell}(z)}{[z^n]M(z,1)}.
\end{align*}
We again apply the Transfer Theorem as we did in the proof of Proposition~\ref{tgn} and obtain that
\begin{align*}
P_{\widehat{m}}
=\frac{1}{12^{\ell+k-s}(\ell-1)!} \left(\frac{\partial^{\ell-1}}{\partial u^{\ell-1}}u^{\ell-1} \frac{a_3(u)}{a_3}\right)\bigg|_{u=1}
=\xi_{\ell}\left(\frac{1}{12}\right)^{k-s},
\end{align*}
which is clearly a positive rational number.
\end{proof}

\section{Pattern occurrences}\label{POPT}

In this section, we consider the global problem for the number of occurrences of a given pattern in a random planar map. This is done by the rerooting method, which extends results from Section~\ref{LPP} concerning occurrences of a given pattern at the root of an arbitrary map.

The idea of rerooting consists in considering all possible choices for the root in those maps that have a pattern occurrence at the root in some bigger map and then mapping the set of these \emph{rerooted} maps bijectively to the set of all maps that have occurrences of these maps (not necessarily at the root). We will explain this concept more clearly in the proof of Proposition~\ref{poct}.

A {\em rerooting} of a map consists of choosing an edge from this map and directing it so that it becomes the root edge of the newly created map.
The difference of ``$P$ is a pattern in $M$" and ``$P$ is a pattern at the root of $M$" is that the first one relaxes the location of the root edge of $M$ from the second one, which means we can obtain the first one by rerooting the $M$ of the second one.

\begin{remark}
In order to keep the last condition $F^*(P)\subseteq F^*(M)$ of ``a pattern $P$ in a map $M$" in Definition~\ref{defpp}, 
we have a limitation of rerooting $M$, namely, the new root face $f$ of the obtained map is not allowed to be an inner-face of $P$, for
otherwise we have $f\in F^*(P)$ and $f\notin F^*(M)$ which is not allowed. 
\end{remark}

In the following examples, the map $P$ is a pattern that occurs at the root of the map $M$, and if we further reroot $M$ by choosing $e_1$ as the new root edge, then we obtain the new map $M_1$ which has the pattern $P$ occurring in it.
In addition, we give also an example of an illegal rerooting, namely, if $e_2$ is chosen as the new root edge, then the new root face of $M_2$ is the inner-face of $P$ which leads to $M_2$ not having $P$ as a pattern.
The reason is that if the new root face is an inner-face of $P$, then
the shape of $P$ is broken when we consider the new root face as outer face.
\begin{center}
\begin{figure}[H]
\begin{minipage}{1\textwidth}
    \begin{center}
        {\includegraphics[width=1\textwidth]{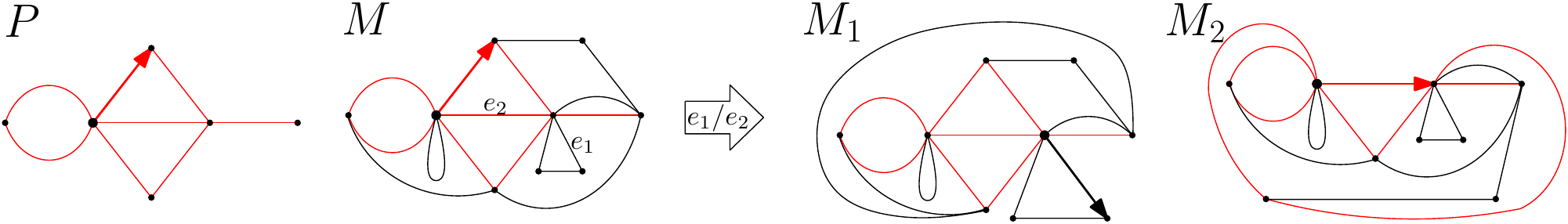}}
        \end{center}
\end{minipage}
\end{figure}
\end{center}

Next, we say that a map $\widehat{m}$ is a {\em marked pattern} in a map $M$ if 
one of the patterns in $M$ isomorphic to $\widehat{m}$ is chosen and marked and its root edge direction is erased.
In other words, since a map $\widehat{m}$ might occur in a map as a pattern many times, we distinguish one of the occurrences and forget the direction of its root edge.
We also say that an edge $e$ is a {\em marked edge} in a map $M$ if $e$ is a distinguished edge in $M$ 
to which a direction is added.

A rerooting of a map $M$ is said to be {\em rotational} if the root face of the rerooted map is the root face of $M$.
Two maps $M$ and $M'$ are {\em rotationally isomorphic} if $M'$ is obtained by a rotational rerooting of $M$ and $M'$ is isomorphic to $M$.

Let $t_{\widehat{m},n}$ be the number of maps with $n$ edges and a marked pattern $\widehat{m}$.
We denote by $\mathbb{T}_{\widehat{m}}(z)$ the ordinary generation function 
\begin{align}\label{tgzd}
\mathbb{T}_{\widehat{m}}(z):=\sum_{n\ge 0}t_{\widehat{m},n}z^n.
\end{align}

\begin{proposition}\label{poct}
For every map $\widehat{m}$, the following equation holds:
\begin{align*}
\mathbb{T}_{\widehat{m}}(z)=|\mathscr{R}_{\widehat{m}}|^{-1}\left(2z\cdot\left(\cfrac{d}{dz}\mathbb{F}_{\widehat{m}}(z)\right)-|\mathscr{I}^{\Sigma}_{\widehat{m}}|\cdot \mathbb{F}_{\widehat{m}}(z)\right).
\end{align*}
where $|\mathscr{I}^{\Sigma}_{\widehat{m}}|$ denotes the sum of valencies of inner-faces of $\widehat{m}$ and $|\mathscr{R}_{\widehat{m}}|$ is the number of rotationally isomorphic maps of $\widehat{m}$.
\end{proposition}

\begin{proof}
We claim that
the sequence $(t_{\widehat{m},n})_{n \ge 0}$ and the sequence $(\tilde{f}_{\widehat{m},n})_{n \ge 0}$ satisfy the following relation:
\[
t_{\widehat{m},n}=|\mathscr{R}_{\widehat{m}}|^{-1}\cdot{\left(2n-|\mathscr{I}^{\Sigma}_{\widehat{m}}|\right)\cdot\tilde{f}_{\widehat{m},n}}.
\]
We denote by $\mu_{\widehat{m},n}$ the number of maps with $n$ edges, a pattern $\widehat{m}$ occurring at the root, and a marked edge 
whose left face is not an inner-face of $\widehat{m}$.

To prove the claim, we first relate the sequences $(\mu_{\widehat{m},n})_{n \ge 0}$ and $(\tilde{f}_{\widehat{m},n})_{n \ge 0}$.
In order to mark an edge in a map, 
we pick one of the $n$ edges of a given map and direct it,
which gives $2n$ possibilities. 
Next, we have to exclude $|\mathscr{I}^{\Sigma}_{\widehat{m}}|$ of them, 
as they correspond to cases when a left face of a marked edge is an inner-face of $\widehat{m}$.
This leads to that for every map $\widehat{m}$, the relation between $\mu_{\widehat{m},n}$ and $\tilde{f}_{\widehat{m},n}$ is
\[
\mu_{\widehat{m},n}=\left(2n-|\mathscr{I}^{\Sigma}_{\widehat{m}}|\right)\cdot\tilde{f}_{\widehat{m},n}.
\]

We next relate sequences $(\mu_{\widehat{m},n})_{n \ge 0}$ and $(t_{\widehat{m},n})_{n \ge 0}$, 
where we recall that for a given map $\widehat{m}$, 
the sequence $(\mu_{\widehat{m},n})_{n\ge 0}$ enumerates:\\
\textbf{(a)} maps with $n$ edges, 
the pattern $\widehat{m}$ occurring at the root,
and a marked edge whose left face is not an inner-face of $\widehat{m}$;\\ 
while the sequence $(t_{\widehat{m},n})_{n\ge 0}$ enumerates:\\
\textbf{(b)} maps with $n$ edges with the marked pattern $\widehat{m}$.

Here we show how to associate $|\mathscr{R}_{\widehat{m}}|$ maps from \textbf{(a)} with one map from \textbf{(b)}.

\paragraph{\textbf{(a)}$\to$\textbf{(b)}}
Let us take a set of maps from class \textbf{(a)} which are equivalent up to
rotational rerooting of the marked pattern $\widehat{m}$.
First, we reroot these maps in such a way that their marked edges become new root edges.
Like that, in each so obtained map the left-hand side of the new root edge is not an
inner-face of $\widehat{m}$.
Next, we erase the direction of the root edge of $\widehat{m}$ in all considered maps.
This gives us exactly one map with pattern $\widehat{m}$ marked.
 
\paragraph{\textbf{(b)}$\to$\textbf{(a)}}
We now reverse the construction. We first root the marked pattern $\widehat{m}$ of a map from class \textbf{(b)}
in order to obtain a map with the pattern $\widehat{m}$ occurring at the root after rerooting.
To do so, we choose a new root edge from edges of the outer face (boundary, clockwise) of the marked pattern 
(we have $|\mathscr{R}_{\widehat{m}}|$ different choices) and reroot the map. 
We thus obtain a map with the pattern $\widehat{m}$ occurring at the root
and a marked edge (the original root edge) whose left face (the original root face) 
is not an inner-face of $\widehat{m}$.

Consequently, we have a one-to-$|\mathscr{R}_{\widehat{m}}|$ correspondence between maps from \textbf{(a)} and \textbf{(b)}.
This leads to that for every map $\widehat{m}$, the relation between
$\mu_{\widehat{m},n}$ and $t_{\widehat{m},n}$ is
\[
\mu_{\widehat{m},n}=|\mathscr{R}_{\widehat{m}}| \cdot t_{\widehat{m},n}
\]

We finish the proof of the claimed result by combining all results above and by rewriting the claim by means of ordinary generation functions.
\end{proof}

\section{Submap occurrences}\label{SOPT}

In this section, we investigate the general problem of enumerating occurrences of a given submap in a random planar map. 
One can obtain ``$S$ is a pattern in $M$" from ``$S$ is a submap at the root of $M$" by rerooting $M$, where this rerooting again has to satisfy the condition that the root face of $M$ is not (or is not a part of) the inner-face of $S$ as was the case in Section~\ref{POPT}.
Let us we mention that, in contrast to computing pattern occurrences, 
we will not enumerate submap occurrences by rerooting so that the submaps occur at the root---another method will be presented in the proof of Proposition~\ref{soct}.

In the following examples, the map $S$ is a submap that occurs at the root of the map $M$, 
and if we further reroot $M$ by choosing $e_1$ as the new root edge, 
then we obtain the new map $M_1$ which has the submap $P$ occurring in it.
In addition, we also give an example of an illegal rerooting, namely, if $e_2$ is chosen as the new root edge, then
$f'$ (the new root face of $M$) is part of $f$ (the inner-face of $S$)
which leads to $M_2$ not having $S$ as a submap.
The reason is that if the new root face is contained in the inner-face of $S$, then
the shape of $S$ is broken when we consider the new root face as outer face.
\begin{center}
\begin{figure}[H]
\begin{minipage}{1\textwidth}
    \begin{center}
        {\includegraphics[width=1\textwidth]{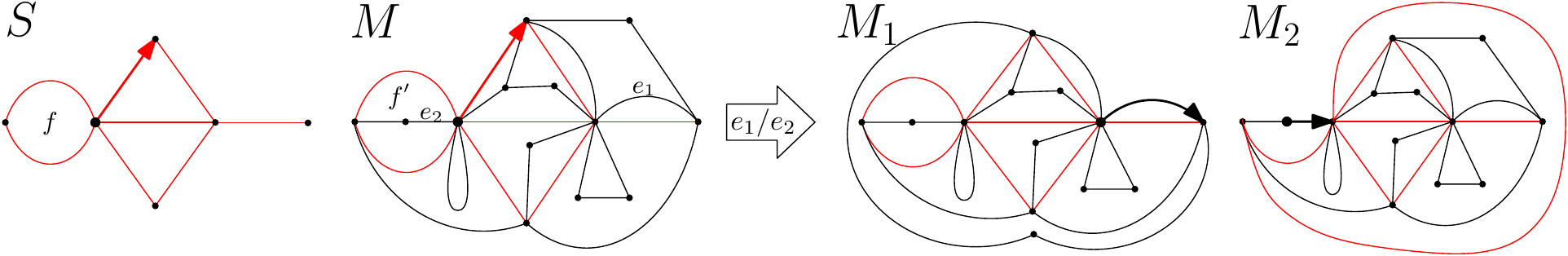}}
        \end{center}
\end{minipage}
\end{figure}
\end{center}

\begin{remark}
In both cases, ``occurrences at the root of a map'' and ``occurrences in a map'', patterns are special cases of submaps: they simply have no additional edges or vertices in their inner-faces.
\end{remark}

Below, this remark will help us to understand the  idea of constructing submaps from patterns.
Here we also mention that a pattern or submap might occur in a map many times. In the following figure, we give three examples of the submap $S$ occuring in the map $M$.
\begin{center}
\begin{figure}[H]
\begin{minipage}{1\textwidth}
    \begin{center}
        {\includegraphics[width=1\textwidth]{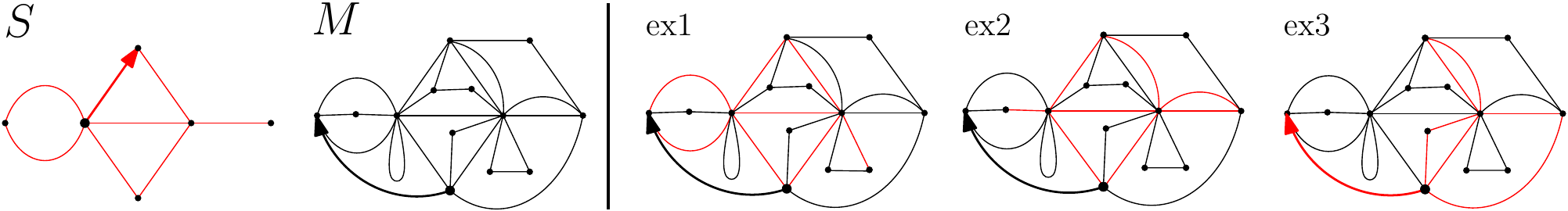}}
        \end{center}
\end{minipage}
\end{figure}
\end{center}

Similar to Section~\ref{POPT}, we consider all possible pairs of the root of a map and a submap in question.
As before, we say that a map $\widehat{m}$ is a {\em marked submap} of a map $M$ if 
one of the submap occurrences $\widehat{m}$ of the map $M$ is chosen and marked
and its root edge direction is erased. 

Let $s_{\widehat{m},n}$ be the number of maps with $n$ edges and a marked submap $\widehat{m}$.
We denoted by $\mathbb{S}_{\widehat{m}}(z)$ the ordinary generation function 
\begin{align}\label{sgzd}
\mathbb{S}_{\widehat{m}}(z):=\sum_{n\ge 0}s_{\widehat{m},n}\,z^n.
\end{align}
A formula for $\mathbb{S}_{\widehat{m}}(z)$ for any given map $\widehat{m}$ is as follows:
\begin{proposition}\label{soct}
Let $\widehat{m}$ be a map, $|\mathscr{I}_{\widehat{m}}|$ be the number of inner-faces of $\widehat{m}$,
and $\Omega_1$, $\Omega_2$, $\cdots$ , $\Omega_{|\mathscr{I}_{\widehat{m}}|}$ be valencies of all inner-faces of $\widehat{m}$. Then the generating function $\mathbb{S}_{\widehat{m}}(z)$ for the sequence enumerating maps with a marked submap $\widehat{m}$ is given by
\begin{align*}
\mathbb{S}_{\widehat{m}}(z)=\mathbb{T}_{\widehat{m}}(z) \cdot
\prod_{i=1}^{|\mathscr{I}_{\widehat{m}}|}\displaystyle z^{-\Omega_i}F_{\Omega_i}(z),
\end{align*}
where $\mathbb{T}_{\widehat{m}}(z)$ is defined in (\ref{tgzd}) and $F_{\ell}(z)$ is defined in (\ref{fmz}).
\end{proposition}

\begin{proof}
We first focus on the relation between the sequence $(s_{\widehat{m},n})_{n \ge 0}$ and the sequence $(t_{\widehat{m},n})_{n \ge 0}$ and claim that
for any map $\widehat{m}$,
\begin{align}\label{formula}
s_{\widehat{m},n}=\sum_{\substack{k,k_1,k_2,\dots,k_{\gamma}\ge 0 \\  k+k_1+k_2+\cdots +k_{\gamma}=n}}\left(t_{\widehat{m},k}\prod_{i=1}^{\gamma}f_{\Omega_i,k_i+\Omega_i}\right),
\end{align}
where we recall that $t_{\widehat{m},n}$ is the number of maps with $n$ edges and a marked pattern $\widehat{m}$, $f_{\ell,n}$ is the number of maps with $n$ edges and their root faces being pure $\ell$-gons, and $|\mathscr{I}_{\widehat{m}}|=\gamma$.

The idea of showing the claim is to construct a map $M$ with a marked submap $\widehat{m}$ by inserting $\gamma$ maps whose outer faces are pure polygons into the inner-face of $\widehat{m}$ where $\widehat{m}$ is a marked pattern of $M$. 
In other words, we can decompose a map $M$ with a marked submap $\widehat{m}$ into a map $M$ with a marked pattern $\widehat{m}$ and $\gamma$ maps whose outer faces are pure polygons.
This construction gives us a way to determine a relation between $s_{\widehat{m},\bullet}$ and $t_{\widehat{m},\bullet}$.

Suppose we have a map with a marked pattern $\widehat{m}$, 
where $\widehat{m}$ has $\gamma$ inner-faces with valencies $\Omega_1$, $\Omega_2$, $\dots$ , $\Omega_{\gamma}$ (order is not important). 
After inserting maps with outer faces being pure $\Omega_1$-gon, pure $\Omega_2$-gon, $\dots$, pure $\Omega_{\gamma}$-gon into those inner-faces, we have a map with a marked submap $\widehat{m}$.

Now, we first compute how many different ways there are when inserting
a random map with $n$ edges and outer face a pure $\ell$-gon
into an existing inner-face of valency $\ell$. 
We have $f_{\ell,n}$ different maps with $n$ edges and root faces pure $\ell$-gons
and we just randomly choose an edge $e$ from this inner-face and insert each of those maps by 
merging $e$ with the root edge of each map and do the same as we did in the proof of Lemma~\ref{bij2}.
Note that it seems we have $\ell$ different ways to insert $M$, but actually it does not matter which edge of inner-face is chosen since of them leads to the same result and covers all cases.

A map with $n$ edges and outer (root) face being a pure $\ell$-gon has $n-\ell$ inner-edges.
Thus, if we want to insert this map into a face, we only need to add $n-\ell$ edges into this face.

If we insert $\gamma$ maps whose outer faces are pure $\Omega_1$-gon, pure $\Omega_2$-gon, $\dots$, pure $\Omega_{\gamma}$-gon and with, respectively, $k_1,k_2,\dots,k_{\gamma}$ inner-edges into a map with $k$ edges and a marked pattern $\widehat{m}$, we get a map with $k+k_1+k_2+\cdots +k_{\gamma}$ edges and a marked submap $\widehat{m}$.

Therefore, since
we have $t_{\widehat{m},k}$ maps with $k$ edges and a marked pattern $\widehat{m}$ and
$f_{\Omega_i,k_i+\Omega_i}$ maps whose outer (root) faces are pure $\Omega_i$-gons with $k_i$ inner-edges,
we obtain that there are 
\[
t_{\widehat{m},k}\prod_{i=1}^{\gamma}f_{\Omega_i,k_i+\Omega_i}
\]
maps with $k+k_1+k_2+\cdots +k_{\gamma}$ edges and a marked submap $\widehat{m}$ whose inner-faces 
contain $k_1,k_2,\dots,k_{\gamma}$ edges, respectively. 
Hence, by collecting all different cases of $k_1,k_2,\dots,k_{\gamma}$,
we finish the proof of the claim. 

Now, we deduce the generating function. From (\ref{formula}), we have
\begin{align*}
\mathbb{S}_{\widehat{m}}(z)=\sum_{n\ge 0}s_{\widehat{m},n}\,z^n=&\,\sum_{n\ge 0}\sum_{\substack{k,k_1,k_2,\dots,k_{\gamma}\ge 0 \\  k+k_1+k_2+\cdots +k_{\gamma}=n}}\left(t_{\widehat{m},k}\prod_{i=1}^{\gamma}f_{\Omega_i,k_i+\Omega_i}\right)\,z^n\\
=&\,\sum_{n\ge 0}\sum_{\substack{k,k_1,k_2,\dots,k_{\gamma}\ge 0 \\  k+k_1+k_2+\cdots +k_{\gamma}=n}}\left(t_{\widehat{m},k}\,z^k\prod_{i=1}^{\gamma}f_{\Omega_i,k_i+\Omega_i}\,z^{k_i}\right)\\
=&\,\left(\sum_{k\ge 0}t_{\widehat{m},k}\,z^k\right)\prod_{i=1}^{\gamma}\left(\sum_{k_i\ge 0}f_{\Omega_i,k_i+\Omega_i}\,z^{k_i}\right).
\end{align*}
Now, by observing that
\[
\sum_{k_i\ge 0}f_{\Omega_i,k_i+\Omega_i}\,z^{k_i}
=z^{-\Omega_i}\sum_{k_i\ge 0}f_{\Omega_i,k_i+\Omega_i}\,z^{k_i+\Omega_i}
=z^{-\Omega_i} F_{\Omega_i}(z)
\]
which comes from the definition of Lemma~\ref{lemfz}, the proof is finished.
\end{proof}

\section{Proof of the main results}\label{PF}

Now we are in the position to prove Theorems~\ref{thmpt} and~\ref{thmsm}.

\begin{proof}[Proof of Theorems ~\ref{thmpt} and~\ref{thmsm}]
The expected value of the number of occurrences of the pattern $\widehat{m}$ in a map with $n$ edges 
can be computed by comparing $t_{\widehat{m},n}$ of (\ref{tgzd}) and $m_n$ of (\ref{m0-gen2}). 
More precisely, for a random map with $n$ edges, the expected value of the number of occurrences
of the pattern $\widehat{m}$ is
\begin{align*}
\mathbb{E}[t(\widehat{m},M_n)]:=\frac{t_{\widehat{m},n}}{m_n}.
\end{align*}
For this purpose, we expand $\mathbb{T}_{\widehat{m}}(z)$ at its dominant singularity $z=1/12$ as follows:
\begin{align}\label{tme12}
\mathbb{T}_{\widehat{m}}(z)=\sum_{i\ge 0}\tau_{\widehat{m},i}\,(1-12z)^{i/2},
\end{align}
and by applying the Transfer Theorem, we obtain that
\begin{align*}
t_{\widehat{m},n}=[z^n]\mathbb{T}_{\widehat{m}}(z)= \,12^n\left(\frac{-\tau_{\widehat{m},1}}{2\sqrt{\pi}}\,n^{-3/2}+\frac{12\,\tau_{\widehat{m},3}-3\,\tau_{\widehat{m},1}}{16\sqrt{\pi}}\,n^{-5/2}+O(n^{-7/2})\right).
\end{align*}
By the well-known result \cite{GJ1983},
\begin{align*}
m_n=[z^n]M(z,1)= \,12^n\left(\frac{2}{\sqrt{\pi}}\,n^{-5/2}-\frac{25}{4\sqrt{\pi}}\,n^{-7/2}+O(n^{-9/2})\right),
\end{align*}
we obtain that
\begin{align}\label{tcc}
\mathbb{E}[t(\widehat{m},M_n)]=\frac{t_{\widehat{m},n}}{m_n}=\frac{-\tau_{\widehat{m},1}}{4}\,n+\frac{3\,\tau_{\widehat{m},3}-7\,\tau_{\widehat{m},1}}{8}+O(1/n).
\end{align}

Similarly, by comparing $s_{\widehat{m},n}$ of (\ref{sgzd}) and $m_n$ of (\ref{m0-gen2}), we know that
for a random map with $n$ edges, the expected number of occurrences of $\widehat{m}$ as a submap is equal to
\begin{align*}
\mathbb{E}[s(\widehat{m},M_n)]:=\frac{s_{\widehat{m},n}}{m_n}.
\end{align*}
Similar to the computation of $\mathbb{E}[t(\widehat{m},M_n)]$, we expand $\mathbb{S}_{\widehat{m}}(z)$ at its dominant singularity $z=1/12$ as follows:
\begin{align}\label{sme12}
\mathbb{S}_{\widehat{m}}(z)=\sum_{i\ge 0}\varrho_{\widehat{m},i}\,(1-12z)^{i/2},
\end{align}
and obtain that
\begin{align*}
s_{\widehat{m},n}=[z^n]\mathbb{S}_{\widehat{m}}(z)
=\,12^n\left(\frac{-\varrho_{\widehat{m},1}}{2\sqrt{\pi}}\,n^{-3/2}+\frac{12\,\varrho_{\widehat{m},3}-3\,\varrho_{\widehat{m},1}}{16\sqrt{\pi}}\,n^{-5/2}+O(n^{-7/2})\right).
\end{align*}
Thus, 
\begin{align}\label{scc}
\mathbb{E}[s(\widehat{m},M_n)]=\frac{s_{\widehat{m},n}}{m_n}=\frac{-\varrho_{\widehat{m},1}}{4}\,n+\frac{3\,\varrho_{\widehat{m},3}-7\,\varrho_{\widehat{m},1}}{8}+O(1/n).
\end{align}

In order to finish the proof, we still need to show that $c_1=-\tau_{\widehat{m},1}/4$ and $c_1'=-\varrho_{\widehat{m},1}/4$ are both positive rational numbers. Thus, we focus on proving that $\tau_{\widehat{m},1}\in \mathbb{Q}^-$ and $\varrho_{\widehat{m},1}\in \mathbb{Q}^-$.

\vspace{0cm}
\paragraph{\em Negativity and rationality of $\tau_{\widehat{m},1}$}
Here we again let $\widehat{m}$ be a map with $k$ inner-edges, $s$ outer-edges, and the outer face of valency $\ell$.
By Proposition~\ref{poct},
\begin{align}\label{tgg}
\mathbb{T}_{\widehat{m}}(z)=|\mathscr{R}_{\widehat{m}}|^{-1}\left(2z\cdot\left(\cfrac{d}{dz}\mathbb{F}_{\widehat{m}}(z)\right)-|\mathscr{I}^{\Sigma}_{\widehat{m}}|\cdot \mathbb{F}_{\widehat{m}}(z)\right),
\end{align}
where $\mathbb{F}_{\widehat{m}}(z)$ by Lemma~\ref{bij2} is given by
\begin{align}\label{fgg}
\mathbb{F}_{\widehat{m}}(z)=z^{k-s}\,F_{\ell}(z).
\end{align}
To analyze $\tau_{\widehat{m},1}$, we first expand $\mathbb{F}_{\widehat{m}}(z)$ at its dominant singularity $z=1/12$ as follows:
\begin{align}\label{tgg2}
\mathbb{F}_{\widehat{m}}(z)=\sum_{i\ge 0}\tilde{\kappa}_{\widehat{m},i}\,(1-12z)^{i/2}.
\end{align}
Using Equation (\ref{fgg}), we can now
compare the coefficients of the power series from Equations (\ref{tgg2}) and (\ref{hm1}).
Since $\kappa_{\ell,1}=0$, we have $\tilde{\kappa}_{\widehat{m},1}=0$. Furthermore,
\begin{align}\label{pss1}
\tilde{\kappa}_{\widehat{m},3}=12^{-k+s}\cdot \kappa_{\ell,3},
\end{align} 
where $\kappa_{\ell,3}$ was computed in (\ref{hm2}).

Next, we compare the coefficients of the power series from Equations (\ref{tme12}) and (\ref{tgg2})
and with help of Equation (\ref{tgg}) we obtain that
\begin{align*}
\tau_{\widehat{m},1}=\frac{-3\,\tilde{\kappa}_{\widehat{m},3}}{|\mathscr{R}_{\widehat{m}}|}.
\end{align*} 
Hence, we get that $\tau_{\widehat{m},1}\in \mathbb{Q}^-$ if and only if $\tilde{\kappa}_{\widehat{m},3}\in \mathbb{Q}^+$ and analogously by (\ref{pss1}), $\tilde{\kappa}_{\widehat{m},3}\in \mathbb{Q}^+$ if and only if $\kappa_{\ell,3}\in \mathbb{Q}^+$.
Furthermore, by (\ref{topo}), we know that $\kappa_{\ell,3}\in \mathbb{Q}^+$ if and only if $\xi_{\ell} \in \mathbb{Q}^+$.  
Finally, by applying Lemma~\ref{posi1}, we obtain that $\tau_{\widehat{m},1}\in \mathbb{Q}^-$.

\vspace{0cm}
\paragraph{\em Negativity and rationality of $\varrho_{\widehat{m},1}$ .}
We know that a pattern is a special case of submap. 
So we have the inequality
$s_{\widehat{m},n}\ge t_{\widehat{m},n}$ which implies that,
by comparing (\ref{tcc}) and (\ref{scc}), $\varrho_{\widehat{m},1}\le \tau_{\widehat{m},1}$. We already know that $\tau_{\widehat{m},1}$ is negative, thus $\varrho_{\widehat{m},1}$ is also negative.

By Proposition~\ref{soct}, we have
\begin{align*}
\mathbb{S}_{\widehat{m}}(z)=\mathbb{T}_{\widehat{m}}(z) \cdot
\prod_{i=1}^{|\mathscr{I}_{\widehat{m}}|}\displaystyle z^{-\Omega_i}F_{\Omega_i}(z).
\end{align*}
Using this formula to compare the coefficients of Equations (\ref{hm1}), (\ref{tme12}) and (\ref{sme12}),
we obtain the following equation:
\begin{align}\label{snn}
\varrho_{\widehat{m},1}=\tau_{\widehat{m},1}\,\prod_{i=1}^{|\mathscr{I}_{\widehat{m}}|}12^{\Omega_i}\,\kappa_{\Omega_i,0}.
\end{align} 
In a similar way in which we obtained Equation (\ref{hm2}) in the proof of Theorem~\ref{tgn}, we can also compute the value of $\kappa_{\ell,0}$:
\begin{align*}
\kappa_{\ell,0}=\frac{1}{12^{\ell}(\ell-1)!}\left(\frac{\partial^{\ell-1}}{\partial u^{\ell-1}}u^{\ell-1}a_0(u)\right)\Bigg|_{u=1},
\end{align*} 
where $a_0(u)$ is as in (\ref{a0ueue}).

Since $\tau_{\widehat{m},1}$ is rational and by (\ref{snn}), we know that $\varrho_{\widehat{m},1}$ is rational if $\kappa_{\ell,0}$ is rational. 
Let us now deal with the rationality of $\kappa_{\ell,0}$. 
Similarly to (\ref{pue}), we have
\[
\sqrt{3(u+2)(-5u+6)^3}=6\left(\sum_{i\ge 0}\binom{1/2}{i}\left(\frac{u}{2}\right)^i\right)
\cdot\left(\sum_{j\ge 0}\binom{3/2}{j}\left(-\frac{5u}{6}\right)^j\right).
\]
By showing that $a_0(u)$ has rational coefficients in its Taylor series expansion at $u=0$, we get that $\kappa_{\ell,0}$ is rational as well.
\end{proof}

\begin{remark}
The constants $c_2$ and $c'_2$ of Theorem~\ref{thmpt} and~\ref{thmsm} are computed in (\ref{tcc}) and (\ref{scc}) as follows:
\begin{align*}
c_2=\frac{3\,\tau_{\widehat{m},3}-7\,\tau_{\widehat{m},1}}{8},\quad \text{and} \quad
c'_2=\frac{3\,\varrho_{\widehat{m},3}-7\,\varrho_{\widehat{m},1}}{8},
\end{align*}
where $\tau_{\widehat{m},3}$ and $\varrho_{\widehat{m},3}$ can be computed by comparing coefficients as well. We have
\[
\tau_{\widehat{m},3}=|\mathscr{R}_{\widehat{m}}|^{-1}\left(-5\,\tilde{\kappa}_{\widehat{m},5}+\left(3-|\mathscr{I}^{\Sigma}_{\widehat{m}}|\right)\,\tilde{\kappa}_{\widehat{m},3}\right)
\]
and
\begin{align*}
\varrho_{\widehat{m},3}=&\,\tau_{\widehat{m},1}\left(\prod_{i=1}^{|\mathscr{I}_{\widehat{m}}|}12^{\Omega_i}\,\kappa_{\Omega_i,0}\right)
\cdot\left(\sum_{i=1}^{|\mathscr{I}_{\widehat{m}}|}\frac{\kappa_{\Omega_i,2}}{\kappa_{\Omega_i,0}}\right)
+\tau_{\widehat{m},3}\left(\prod_{i=1}^{|\mathscr{I}_{\widehat{m}}|}12^{\Omega_i}\,\kappa_{\Omega_i,0}\right).
\end{align*}
\end{remark}



\section{Example}\label{EX}

In this section, 
we give an example to make it easier for readers to understand how to apply the main results of this paper.

Let $\widehat{m}$ be a map with $V(\widehat{m})=\{A,B,C,D\}$, $E(\widehat{m})=\{\overline{AB},\overline{BC},\overline{CD},\overline{DA},\overline{AC}\}$ and $\overrightarrow{AB}$ is the root edge
(see $\widehat{m}$ in Figure~\ref{pat_exp}).
We first deal with $\widehat{m}$ occurring as a pattern in maps. By knowing that $|\mathscr{I}^{\Sigma}_{\widehat{m}}|$, which denotes the sum of valencies of inner-faces of $\widehat{m}$, is $6$ and
$|\mathscr{R}_{\widehat{m}}|$, which denotes the number of rotationally isomorphic maps of $\widehat{m}$, is $2$,
together with applying Proposition~\ref{poct}, we have
\[
\mathbb{T}_{\widehat{m}}(z)=\frac{1}{2}\left(2z\left(\cfrac{d}{dz}\mathbb{F}_{\widehat{m}}(z)\right)-6\mathbb{F}_{\widehat{m}}(z)\right),
\]
where one can apply
Lemma~\ref{bij2} and obtain that
$\mathbb{F}_{\widehat{m}}(z)=zF_4(z)$, which leads to
\[
\mathbb{T}_{\widehat{m}}(z)=z\left(\cfrac{d}{dz}zF_4(z)\right)-3zF_4(z).
\]
Then, by Lemma~\ref{lemfz}, we obtain the explicit result
\[
\mathbb{T}_{\widehat{m}}(z)=\frac{s(z)+t(z)\sqrt{1-12z}}{177147z^3},
\]
with
\begin{align*}
s(z)&:=-6804z^5-324z^4+1998z^3+36z^2-75z+5,\\
t(z)&:=486z^5-1458z^4-864z^3+144z^2+45z-5.
\end{align*}

If we are interested in the number of maps with $n$ edges and a marked pattern $\widehat{m}$,
we can do Taylor series expansion at $z = 0$ as follows:
\begin{align}\label{eqqt}
\mathbb{T}_{\widehat{m}}(z)=2z^5+42z^6+632z^7+8380z^8+\cdots,    
\end{align}
and $[z^n]\mathbb{T}_{\widehat{m}}(z)$ is the corresponding answer (see example in Figure~\ref{pat_exp} below and the Appendix).

\begin{center}
\begin{figure}[H]
\begin{minipage}{1\textwidth}
    \begin{center}
        {\includegraphics[width=0.5\textwidth]{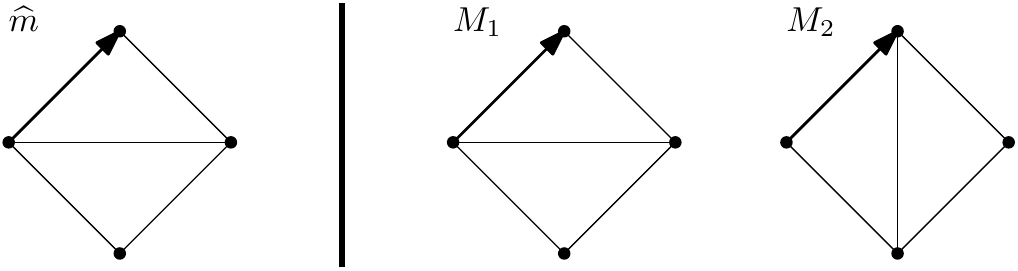}}
        \end{center}
\end{minipage}
\caption{$[z^5]\mathbb{T}_{\widehat{m}}(z)=2$ means that there are two maps ($M_1$ and $M_2$) with $5$ edges and a marked pattern $\widehat{m}$.}
\label{pat_exp}
\end{figure}
\end{center}

And if we further want to know the asymptotic behavior of $[z^n]\mathbb{T}_{\widehat{m}}(z)$, we can do 
 singular expansion of the function $\mathbb{T}_{\widehat{m}}(z)$
at its dominant singularity $1/12$,
\[
\mathbb{T}_{\widehat{m}}(z)=\frac{29}{26244}-\frac{419}{52488}(1-12z)^{1/2}+\frac{361}{13122}(1-12z)+\cdots,
\]
and by applying the Transfer Theorem, we have
\[
t_{{\widehat{m}},n}=[z^n]\mathbb{T}_{\widehat{m}}(z)\sim \frac{-419\, n^{-3/2}\,12^n}{52488\,\Gamma(-1/2)}.
\]
Thus, by comparing with (\ref{ma3b}), we obtain
\[
\mathbb{E}[t(\widehat{m},M_n)]=\frac{t_{{\widehat{m}},n}}{m_n}\sim \frac{419n}{209952}\simeq 0.002n,
\]
which satisfies the conclusion of Theorem~\ref{thmpt}.

Finally, as for the occurrences of $\widehat{m}$ as a submap in maps, 
by knowing that two inner faces of $\widehat{m}$ are both triangles, we set $\Omega_1=3$ and $\Omega_2=3$
and apply Proposition~\ref{soct}. This gives 
\[
\mathbb{S}_{\widehat{m}}(z)=\mathbb{T}_{\widehat{m}}(z)z^{-6}F_3^2,
\]
which has singular expansions at its dominant singularity $1/12$,
\[
\mathbb{S}_{\widehat{m}}(z)=\frac{118784}{4782969}-\frac{858112}{4782969}(1-12z)^{1/2}+\frac{641024}{4782969}(1-12z)+\cdots,
\]
and, again by the Tansfer Theorem, its asymptotic behavior is given as
\[
s_{{\widehat{m}},n}=[z^n]\mathbb{S}_{\widehat{m}}(z)\sim \frac{-858112\, n^{-3/2}\,12^n}{4782969\,\Gamma(-1/2)}.
\]
Finally, we know that the expected number of occurrences of $\widehat{m}$ as a submap in a map is 
\[
\mathbb{E}[s(\widehat{m},M_n)]=\frac{s_{{\widehat{m}},n}}{m_n}\sim \frac{214528n}{4782969}\simeq 0.045n,
\]
which satisfies the conclusion of Theorem~\ref{thmsm}.

\section*{Acknowledgment}

The author would like to thank Prof. Michael Drmota for guidance and a lot of precious input. He also thanks Prof. Michael Fuchs and Dr. Katarzyna Grygiel for their valuable comments on a previous version of this paper. The author was supported by the Austrian Science Fund FWF, Project F50-02 and F50-10 that are part of the SFB ``Algorithmic and Enumerative Combinatorics''.

\section*{Appendix}

In (\ref{eqqt}), $[z^6]\mathbb{T}_{\widehat{m}}(z)=42$ means there are 42 maps with $6$ edges and a marked pattern $\widehat{m}$. The following figure lists all cases of them where $\widehat{m}$ is painted red.
\begin{center}
\begin{figure}[H]
\begin{minipage}{1\textwidth}
    \begin{center}
        {\includegraphics[width=0.9\textwidth]{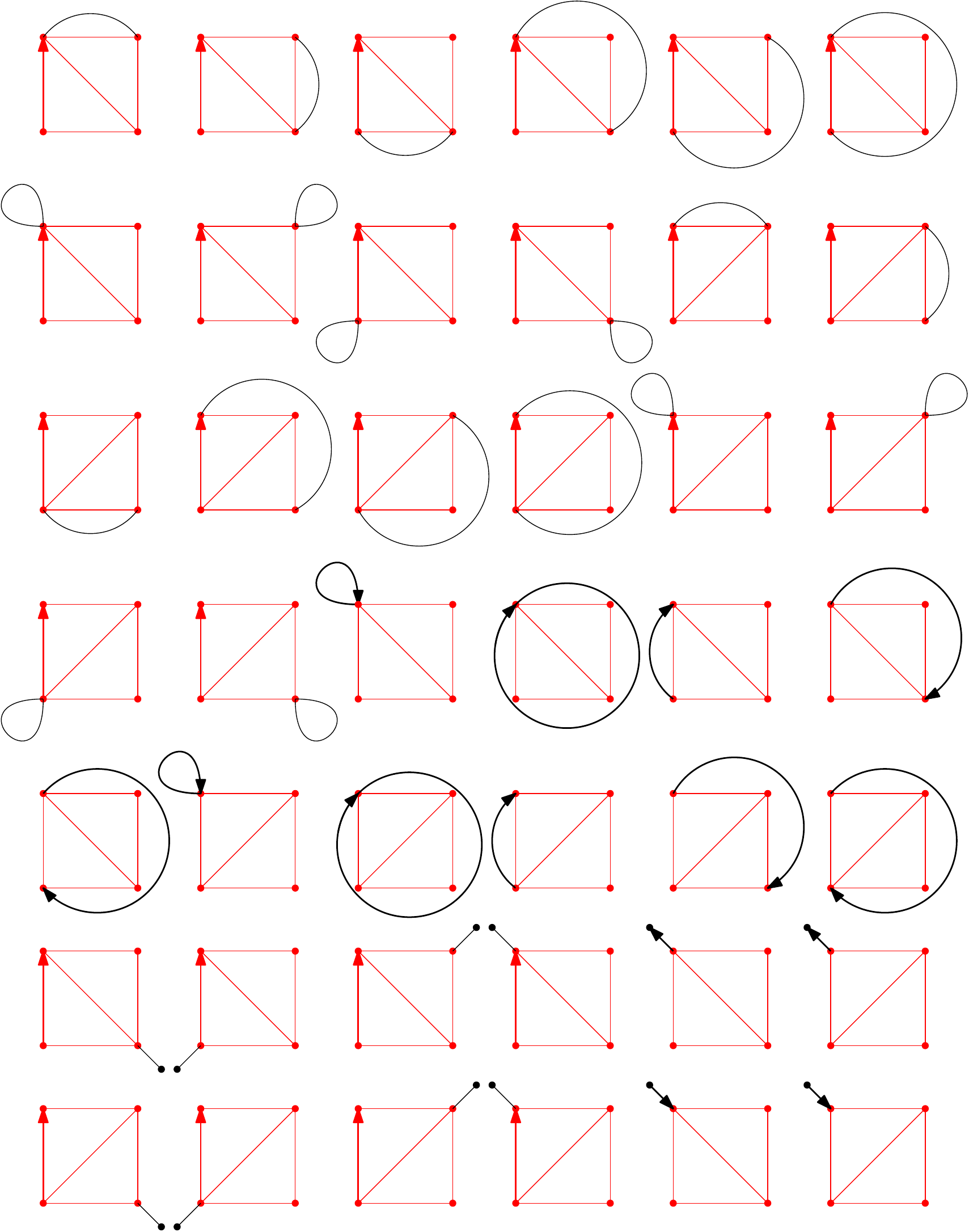}}
        \end{center}
\end{minipage}
\end{figure}
\end{center}
\end{document}